\documentclass[10pt,a4paper,reqno]{amsart}
\usepackage[utf8]{inputenc}
\usepackage[T1]{fontenc}       
\usepackage{ae}   
\usepackage{amsfonts}         
\usepackage{amsmath}
\usepackage{amsthm}
\usepackage{amssymb}
\usepackage{mathtools}
\usepackage{calc}
\usepackage{centernot}
\usepackage{color}
\usepackage[pdftex]{hyperref}
\usepackage{mathrsfs}
\usepackage[marginpar]{todo}

\newcommand{\ve}{\varepsilon}

\newcommand{\R}{\mathbb{R}}

\newcommand{\Rn}{\mathbb{R}^n}

\newcommand{\Ss}{\mathbb{S}}
\newcommand{\abs}[1]{\lvert #1 \rvert}

\newcommand{\ang}[1]{\langle #1 \rangle}

\newcommand{\QQ}{\mathcal{Q}}

\newcommand{\U}{\mathcal{U}}

\newcommand{\F}{\mathcal{F}}

\newcommand{\A}{\mathcal{A}}
\newcommand{\p}{\partial}

\newcommand{\dv}{\text{div }}

\newcommand{\inb}{\partial_{\infty}}
\newcommand{\an}{\sphericalangle}
\newcommand{\pinf}{\partial_{\infty}}
\newcommand{\bM}{\bar{M}}
\newcommand{\tve}{\tilde{\varepsilon}}
\newcommand{\LWp}{W_{\text{loc}}^{1,p}}
\newcommand{\Wop}{W_0^{1,p}}
\DeclareMathOperator*{\esssup}{ess\,sup}

\numberwithin{equation}{section}

\theoremstyle{plain}
\newtheorem{thm}{Lause}[section]
\newtheorem{theo}[thm]{Theorem}
\newtheorem{lem}[thm]{Lemma}
\newtheorem{cor}[thm]{Corollary}
\newtheorem{prop}[thm]{Proposition}

\theoremstyle{definition}
\newtheorem{defin}[thm]{Definition}

\begin{document}

\title[Asymptotic Dirichlet problem]{Asymptotic Dirichlet problem for $\A$-harmonic functions
on manifolds with pinched curvature}



\author{Esko Heinonen}
\address{E. Heinonen, Department of Mathematics and Statistics, P.O.B. 68 (Gustaf
H\"alstr\"omin katu 2b), 00014 University
of Helsinki, Finland.}
\email{esko.heinonen@helsinki.fi}

\thanks{The author was supported by the Academy of Finland, project 252293, and the Wihuri Foundation.}




\begin{abstract}
We study the asymptotic Dirichlet problem for $\A$-harmonic functions on a Cartan-Hadamard
manifold whose radial sectional curvatures outside a compact set satisfy an upper bound
   $$
     K(P)\le - \frac{1+\ve}{r(x)^2 \log r(x)}
   $$
and a pointwise pinching condition
   $$
     \abs{K(P)}\le C_K\abs{K(P')}
   $$
for some constants $\ve>0$ and $C_K\ge 1$, where $P$ and $P'$ are any 2-dimensional
subspaces of $T_xM$ containing the (radial) vector $\nabla r(x)$ and $r(x)=d(o,x)$ is the
distance to a fixed point $o\in M$. We solve the asymptotic Dirichlet problem with any 
continuous boundary data $f\in C(\pinf M)$. The results apply also to the Laplacian and 
$p$-Laplacian, $1<p<\infty,$ as special cases.
\end{abstract}

\maketitle

\section{Introduction}

\noindent In this paper we are interested in the asymptotic Dirichlet problem for 
$\A$-harmonic functions on a Cartan-Hadamard manifold $M$ of dimension $n\ge 2$. 
We recall that a Cartan-Hadamard
manifold is a simply connected complete Riemannian manifold with non-positive sectional 
curvature. Since the exponential map $\exp_o \colon T_oM \to M$ is a diffeomorphism for 
every point $o\in M$, it follows that $M$ is diffeomorphic to $\Rn$. One can define an 
asymptotic boundary $\pinf M$ of $M$ as the set of all equivalence classes of unit speed
geodesic rays on $M$. Then the compactification of $M$ is given by $\bM = M\cup \pinf M$
equipped with the \emph{cone topology}. We also notice that $\bM$ is homeomorphic to the 
closed Euclidean unit ball; for details, see Section \ref{prelim} and \cite{eberleinoneill}.

The \emph{asymptotic Dirichlet problem} on $M$ for some
operator $\QQ$ is the following: Given a function $f \in C(\pinf M)$ does there exist a
(unique) function $u \in C(\bM)$ such that $\QQ[u]=0$ on $M$ and $u|\pinf M = f$?
We will consider this problem for the $\A$-harmonic operator (of type $p$)
 \begin{equation}\label{aharmoper}
  \QQ[u] = -\dv \A_x(\nabla u),
 \end{equation}
where $\A \colon TM \to TM$ is subject to certain conditions; for instance $\ang{\A(V),V}
\approx |V|^p, \, 1<p<\infty$, and $\A(\lambda V) = \lambda |\lambda|^{p-2} \A(V)$ for 
all $\lambda \in \R\setminus \{0\}$ (see Section \ref{aharmfunc} for the precise definition).
A function $u$ is said to be $\A$-\emph{harmonic} if it satisfies the equation
  \begin{equation}\label{aharmeq}
   -\dv \A_x(\nabla u) = 0.
  \end{equation}

The asymptotic Dirichlet problem on Cartan-Hadamard manifolds has been solved for various
operators and under various assumptions on the manifold. The first result for this problem 
was due to Choi \cite{choi} when he solved the asymptotic Dirichlet problem for the Laplacian 
assuming that the sectional curvature has a negative upper bound $K_M \le -a^2 <0$, and that  
any two points at infinity can be separated by convex neighborhoods. Anderson \cite{anderson}
showed that such convex sets exist provided the sectional curvature of the manifold satisfies 
$-b^2\le K_M \le -a^2 <0$. We point out that Sullivan \cite{sullivan} solved independently the 
asymptotic Dirichlet problem for the Laplacian under the same curvature assumptions but
using probabilistic arguments. Cheng \cite{cheng} was the first to solve the problem for the 
Laplacian under the
same type of pointwise pinching assumption for the sectional curvatures as we consider in this paper.
Later the asymptotic Dirichlet problem has been generalized for $p$-harmonic and 
$\mathcal{A}$-harmonic functions and for minimal graph equation under various curvature 
assumptions, see 
\cite{casterasheinonenholopainen}, \cite{casterasholopainenripoll2}, \cite{holopainenplaplace}, 
\cite{holopainenvahakangas}, \cite{vahakangaspinch}, \cite{vahakangasyoung}.

In \cite{vahakangaspinch} V\"ah\"akangas had exactly the same pinching condition but with
weaker upper bound for the sectional curvatures. Namely, he solved the asymptotic Dirichlet
problem assuming the pointwise pinching condition and
  $$
     K(P)\le - \frac{\phi(\phi-1)}{r(x)^2},
  $$
where $\phi>1$ is constant. In \cite{casterasheinonenholopainen} the authors showed that, 
with these weaker assumptions, the solvability result holds also for the minimal graph equation.

In this paper we will use similar techniques as in \cite{casterasheinonenholopainen}, 
\cite{casterasholopainenripoll2} and \cite{vahakangasyoung}.
Our main theorem is the following.

\begin{theo}\label{mainthm}
Let $M$ be a Cartan-Hadamard manifold of dimension $n\ge2$. Assume that
\begin{equation}\label{curvupbound}
K(P) \le -\frac{1+\ve}{r(x)^2\log r(x)},
\end{equation}
for some constant $\ve>0$, where $K(P)$ is the sectional curvature of any two-dimensional
subspace $P\subset T_xM$ containing the radial vector $\nabla r(x)$, with $x\in M\setminus
B(o,R_0)$. Suppose also that there exists a constant $C_K<\infty$ such that
\begin{equation}\label{pinchcond}
|K(P)| \le C_K |K(P')|
\end{equation}
whenever $x\in M\setminus B(o,R_0)$ and $P,P'\subset T_xM$ are two-dimensional subspaces
containing the radial vector $\nabla r(x)$. Then the asymptotic Dirichlet problem for the
$\mathcal{A}$-harmonic equation \eqref{aharmeq} is uniquely solvable for any
boundary data $f\in C(\pinf M)$ provided that $1 < p < n\alpha/\beta$.
\end{theo}

In the case of usual Laplacian we have $\alpha=\beta=1$ and $p=2$. Hence we obtain the 
following special case.

\begin{cor}
 Let $M$ be a Cartan-Hadamard manifold of dimension $n\ge3$ and assume that the assumptions
 \eqref{curvupbound} and \eqref{pinchcond} are satisfied. Then the asymptotic Dirichlet
 problem for the Laplace operator is uniquely solvable for any
boundary data $f\in C(\pinf M)$.
\end{cor}

\subsection*{Acknowledgements} The author would like to thank his advisor Ilkka Holopainen for 
reading the manuscript and making helpful suggestions.

\section{Preliminaries} \label{prelim}

\subsection{Cartan-Hadamard manifolds}
Recall that a Cartan-Hadamard manifold is a complete and simply connected Riemannian
manifold with non-positive sectional curvature. Let $M$ be a Cartan-Hadamard manifold and
$\inb M$ the sphere at infinity, then we denote $\bM=M\cup\inb M$. The sphere at infinity
is defined as the set of all equivalence classes of unit speed geodesic rays in $M$; two
such rays $\gamma_1$ and $\gamma_2$ are equivalent if
   $$
     \sup_{t\ge0} d\big(\gamma_1(t),\gamma_2(t)\big) < \infty.
   $$
The equivalence class of $\gamma$ is denoted by $\gamma(\infty)$. For each $x\in M$ and
$y\in \bM\setminus\{x\}$ there exists a unique unit speed geodesic $\gamma^{x,y}\colon
\R \to M$ such that $\gamma^{x,y}(0)=x$ and $\gamma^{x,y}(t)=y$ for some $t\in(0,\infty]$.
For $x\in M$ and $y,z\in \bM\setminus\{x\}$ we denote by
   $$
     \an_x(y,z)=\an(\dot{\gamma}_0^{x,y},\dot{\gamma}_0^{x,z})
   $$
the angle between vectors $\dot{\gamma}_0^{x,y}$ and $\dot{\gamma}_0^{x,z}$ in $T_xM$.
If $v\in T_xM\setminus\{0\}$, $\alpha>0$, and $R>0$, we define a cone
   $$
     C(v,\alpha) = \{y\in \bM\setminus\{x\} \colon \an(v,\dot{\gamma}_0^{x,y})<\alpha\}
   $$
and a truncated cone
   $$
     T(v,\alpha,R) = C(v,\alpha)\setminus \bar{B}(x,R).
   $$
All cones and open balls in $M$ form a basis for the cone topology in $\bM$. With this
topology $\bM$ is homeomorphic to the closed unit ball $\bar{B}^n\subset \Rn$ and
$\inb M$ to the unit sphere $\Ss^{n-1} = \partial B^n$. For detailed study on the cone
topology,
see \cite{eberleinoneill}.

Let us recall that the local Sobolev inequality holds on any Cartan-Hadamard manifold $M$.
More precisely, there exist constants $r_S>0$ and $C_S<\infty$ such that
  \begin{equation}\label{sobieq}
   \left(\int_B \abs{\eta}^{n/(n-1)}\right)^{(n-1)/n} \le C_S \int_B \abs{\nabla \eta}
  \end{equation}
holds for every ball $B=B(x,r_S) \subset M$ and every function $\eta \in C_0^{\infty}(B)$.
This inequality can be obtained e.g. from Croke's estimate of the isoperimetric constant,
see \cite{cheegergromovtaylor} and \cite{croke}.

\subsection{Jacobi equation}
If $k\colon [0,\infty) \to [0,\infty)$ is a smooth function, we denote by $f_k\in C^{\infty}
\big([0,\infty)\big)$ the solution to the initial value problem
   \begin{equation}\label{jacobieq}
     \left\{
       \begin{aligned}
         f_k''  &= k^2f_k \\
         f_k(0) &= 0, \\
         f_k'(0) &= 1.
       \end{aligned}\right.
   \end{equation}
The solution is a non-negative smooth function. Concerning the curvature upper bound
\eqref{curvupbound}, we have the following estimate by Choi.

\begin{prop}\cite[Prop. 3.4]{choi}\label{choiprop}
Suppose that $f\colon [R_0,\infty) \to \R$, $R_0>0$, is a positive strictly increasing
function satisfying the equation $f''(r) = a^2(r) f(r)$, where
$$
a^2(r) \ge \frac{1+\ve}{r^2\log r},
$$
for some $\ve>0$ on $[R_0,\infty)$. Then for any $0<\tilde{\ve}<\ve$, there exists $R_1>R_0$
such that, for all $r\ge R_1$,
$$
f(r) \ge r(\log r)^{1+\tilde{\ve}}, \quad \frac{f'(r)}{f(r)} \ge \frac{1}{r} +
\frac{1+\tilde{\ve}}{r \log r}.
$$
\end{prop}

The pinching condition for the sectional curvatures gives a relation between the
 maximal and minimal moduli of Jacobi fields along a given geodesic that contains the radial
 vector:
 \begin{lem}\cite[Lemma 3.2]{cheng}\cite[Lemma 3]{vahakangaspinch}\label{vahakpinclemma3}
    Let $v\in T_oM$ be a unit vector and $\gamma=\gamma^v$. Suppose that $r_0>0$ and $k<0$
    are constants such that $K_M(P)\ge k$ for every two-dimensional subspace $P\subset T_xM$,
    $x\in B(o,r_0)$. Suppose that there exists a constant $C_K<\infty$ such that
      $$
	\abs{K_M(P)}\le C_K\abs{K_M(P')}
      $$
    whenever $t\ge r_0$ and $P,P'\subset T_{\gamma(t)}M$ are two-dimensional subspaces containing
    the radial vector $\dot{\gamma}_t$. Let $V$ and $\bar{V}$ be two Jacobi fields along $\gamma$
    such that $V_0=0=\bar{V}_0$, $V_0'\bot \dot{\gamma}_0 \bot \bar{V}_0$,  and $\abs{V_0'}=1=
    \abs{\bar{V}_0'}$. Then there exists a constant $c_0=c_0(C_K,r_0,k)>0$ such that
      $$
	\abs{V_r}^{C_K} \ge c_0\abs{\bar{V}_r}
      $$
    for every $r\ge r_0$.
 \end{lem}
 
 To prove the solvability of the $\A$-harmonic equation, we will need an estimate for the gradient
 of a certain angular function. This estimate can be obtained in terms of Jacobi fields:
 \begin{lem}\cite[Lemma 2]{vahakangaspinch}\label{anglegrad}
    Let $x_0\in M\setminus \{o\}, \, U=M\setminus \gamma^{o,x_0}(\R)$, and define $\theta
    \colon U \to [0,\pi], \, \theta(x) = \an_o(x_0,x)\coloneqq \arccos\ang{\dot{\gamma}_0^{o,x_0},
    \dot{\gamma}_0^{o,x}}$. Let $x\in U$ and $\gamma=\gamma^{o,x}$. Then there exists a 
    Jacobi field $W$ along $\gamma$ with $W(0)=0,\, W_0' \bot \dot{\gamma}_0,$ and $\abs{W_0'}=1$
    such that
      $$
	\abs{\nabla \theta(x)} \le \frac{1}{\abs{W(r(x))}}.
      $$
 \end{lem}

\subsection{$\mathcal{A}$-harmonic functions}\label{aharmfunc}

Let $M$ be a Riemannian manifold and $1<p<\infty$. Suppose that $\A\colon TM\to TM$ is an
operator that satisfies the following assumptions for some constants $0<\alpha\le\beta<\infty$:
the mapping $\A_x=\A|T_xM \colon T_xM \to T_xM$ is continuous for almost every $x\in M$ and
the mapping $x\mapsto \A_x(V_x)$ is measurable for all measurable vector fields $V$ on $M$;
for almost every $x\in M$ and every $v\in T_xM$ we have
\begin{align*}
\ang{\A_x(v),v} &\ge \alpha |v|^p, \\
|\A_x(v)| &\le \beta |v|^{p-1}, \\
\ang{\A_x(v) - &\A_x(w),v-w} >0,
\end{align*}
whenever $w\in T_xM \setminus \{v\}$, and
$$
\A_x(\lambda v) = \lambda |\lambda|^{p-2} \A_x(v)
$$
for all $\lambda \in \R\setminus\{0\}$. The set of all such operators is denoted by
$\A^p(M)$ and we say that $\A$ is of type $p$. The constants $\alpha$ and $\beta$ are
called the \emph{structure constants} of $\A$.

Let $\Omega\subset M$ be an open set and $\A\in \A^p(M)$. A function $u\in C(\Omega) \cap
\LWp(\Omega)$ is \emph{$\A$-harmonic} in $\Omega$ if it is a weak solution of the equation
\begin{equation}\label{AHarm}
-\dv \A(\nabla u) =0.
\end{equation}
In other words, if
\begin{equation}\label{WAHarm}
\int_{\Omega} \ang{\A(\nabla u), \nabla\varphi} = 0
\end{equation}
for every test function $\varphi\in C_0^{\infty}(\Omega)$. If $|\nabla u| \in L^p(\Omega)$,
then it is equivalent to require \eqref{WAHarm} for all $\varphi\in \Wop(\Omega)$ by
approximation.

In the special case $\A(v) = |v|^{p-2} v$, $\A$-harmonic functions are called
\emph{$p$-harmonic} and, in particular, if $p=2$, we obtain the usual harmonic functions.

A lower semicontinuous function $u \colon \Omega \to (-\infty,\infty]$ is called
\emph{$\A$-superharmonic} if $u \not\equiv \infty$ in each component of $\Omega$, and for
each open $D\subset\subset \Omega$ and for every $h\in C(\bar{D})$, $\A$-harmonic in $D$,
$h\le u$ on $\p D$ implies $h\le u$ in $D$.

The \emph{asymptotic Dirichlet problem} (for $\A$-harmonic functions) is the following:
for given function $f\in C(\pinf M)$, find a function $u\in C(\bM)$ such that $\A(u)=0$ in
$M$ and $u|\pinf M = f.$ The asymptotic Dirichlet problem can be solved using the so called
Perron's method which we will recall next. The definitions follow 
\cite{heinonenkilpelainenmartio}.

Fix $p\in(1,\infty)$ and let $\A\in\A^p(M)$.
\begin{defin}
A function $u\colon M\to (-\infty,\infty]$ belongs to the upper class $\U_f$ of $f\colon
\pinf M \to [-\infty,\infty]$ if
\begin{enumerate}
 \item $u$ is $\A$-superharmonic in $M$,
 \item $u$ is bounded from below, and
 \item $\liminf_{x\to x_0} u(x) \ge f(x_0)$ for all $x_0\in\pinf M$.
\end{enumerate}
The function $$\overline{H}_f = \inf\{u\colon u\in \U_f\}$$ is called the \emph{upper Perron
solution} and $\underline{H}_f = -\overline{H}_{-f}$ the \emph{lower Perron solution}.
\end{defin}

\begin{theo}
One of the following is true:
\begin{enumerate}
 \item $\overline{H}_f$ is $\A$-harmonic in $M$,
 \item $\overline{H}_f \equiv \infty$ in $M$,
 \item $\overline{H}_f \equiv -\infty$ in $M$.
\end{enumerate}
\end{theo}
We define $\A$-regular points as follows.
\begin{defin}
A point $x_0\in\pinf M$ is called $\A$-regular if
$$
 \lim_{x\to x_0} \overline{H}_f(x) = f(x_0)
$$
for all $f\in C(\pinf M)$.
\end{defin}

Regularity and solvability of the Dirichlet problem are related. Namely, the asymptotic
Dirichlet problem for $\A$-harmonic functions is uniquely solvable if and only if every point
at infinity is $\A$-regular.

\subsection{Young functions}
Let $\phi\colon [0,\infty) \to [0,\infty)$ be a homeomorphism and let $\psi=\phi^{-1}$.
Define \emph{Young functions} $\Phi$ and $\Psi$ by setting
   $$
     \Phi(t) = \int_0^t \phi(s) \, ds
   $$
and
   $$
     \Psi(t) = \int_0^t \psi(s) \, ds
   $$
for each $t\in[0,\infty)$. Then we have the following \emph{Young's inequality}
   $$
     ab\le \Phi(a) + \Psi(b)
   $$
for all $a,b\in[0,\infty)$. The functions $\Phi$ and $\Psi$ are said to form a
\emph{complementary
Young pair}. Furthermore, $\Phi$ (and similarly $\Psi$) is a continuous, strictly
increasing,
and convex function satisfying
   $$
     \lim_{t\to0^+} \frac{\Phi(t)}{t} = 0
   $$
and
   $$
     \lim_{t\to\infty} \frac{\Phi(t)}{t} = \infty.
   $$
For a more general definition of Young functions see e.g. \cite{kufneroldrichfucik}.

As in \cite{vahakangasyoung}, we consider complementary Young pairs of a special type. For
that, suppose that a homeomorphism $G\colon [0,\infty) \to [0,\infty)$ is a Young function
that is a diffeomorphism on $(0,\infty)$ and satisfies
   \begin{equation}\label{young1}
    \int_0^1 \frac{dt}{G^{-1}(t)} < \infty
   \end{equation}
and
   \begin{equation}\label{young2}
    \lim_{t\to0} \frac{tG'(t)}{G(t)} = 1.
   \end{equation}
Then $G(\cdot^{1/p})^p,\,p>1,$ is also a Young function and we define $F\colon
[0,\infty) \to [0,\infty)$ so that $G(\cdot^{1/p})^p$ and $F(\cdot^{1/p})$ form a complementary
Young pair. The space of such functions $F$ will be denoted by $\F_p$. Note that if $F \in\F_p$,
then also $\lambda F\in \F_p$ and $F(\lambda \cdot) \in \F_p$ for every $\lambda>0$. In
\cite{vahakangasyoung} it is proved that for fixed $\ve_0\in(0,1)$ there exists $F\in\F_p$
such that
   \begin{equation}\label{fexists}
     F(t) \le t^{p+\ve_0} \exp\left( -\tfrac{1}{t}\Big(\log\Big(e + \tfrac{1}{t}\Big)
     \Big)^{-1-\ve_0} \right)
   \end{equation}
for all $t\in [0,\infty)$.

\section{Solving the asymptotic Dirichlet problem}

\noindent In order to solve the asymptotic Dirichlet problem for the $\A$-harmonic equation, we need
the following two lemmas, which we state without proofs. Their proofs can be found from the 
original papers. The first lemma allows us to estimate the supremum of a function in a ball 
by the integral over a bigger ball. The second lemma shows that we can estimate the previous
integral up to another integral, which will be uniformly bounded provided the sectional 
curvatures of $M$ satisfy \eqref{curvupbound} and \eqref{pinchcond}.

\begin{lem}\cite[Lemma 2.15]{vahakangasyoung}
Suppose that $||\theta||_{L^{\infty}} \le 1$. Suppose that $s\in (0,r_S)$ is a constant and $x\in M$. Assume
also that $u \in \LWp(M)$ is a function that is $\A$-harmonic in the open set $\Omega \cap B(x,s)$,
satisfies $u-\theta \in \Wop (\Omega)$, $\inf_M \theta \le u \le \sup_M \theta$, and $u=\theta$ a.e.
in $M \setminus \Omega$. Then
$$
\esssup_{B(x,s/2)}
\varphi \big( |u-\theta | \big)^{p(n-1)} \le c \int_{B(x,s)} \varphi \big( | u-\theta| \big)^p,
$$
where the constant $c$ is independent of $x$.
\end{lem}

\begin{lem}\cite[Lemma 16]{casterasholopainenripoll2}\label{phiintfinlem}
Let $M$ be a Cartan-Hadamard manifold of dimension $n\ge 2$. Suppose that
$$
K(P) \le -\frac{1+\ve}{r(x)^2 \log r(x)},
$$
for some constant $\ve>0$, where $K(P)$ is the sectional curvature of any plane $P \subset T_xM$
that contains the radial vector $\nabla r(x)$ and $x$ is any point in $M\setminus B(o,R_0).$ 
Suppose that $U \subset M$
is an open relatively compact set and that $u$ is an $\A$-harmonic function in $U$, with
$u-\theta \in W_0^{1,p}(U)$, where $\A\in \A^p(M)$ with
$$
1<p<\frac{n\alpha}{\beta},
$$
and $\theta \in W^{1,\infty}(M)$ is a continuous function with $||\theta||_{\infty} \le 1$. Then
there exists a bounded $C^1$-function $\mathcal{C}\colon [0,\infty) \to [0,\infty)$ and a constant
$c_0\ge1$, that is independent of $\theta,U$ and $u$, such that
\begin{align}\label{phiestimF}
\int_U \varphi \big(  |u-\theta|/c_0 &\big)^p \big( \log(1+r) + \mathcal{C}(r) \big) \nonumber \\
&\le c_0 + c_0 \int_U F \left( \frac{c_0 |\nabla \theta| r \log(1+r)}{\log(1+r)+\mathcal{C}(r)} \right)
\big( \log(1+r) + \mathcal{C}(r) \big).
\end{align}
\end{lem}

In what follows, we will denote by $j(x)$ the infimum, and by $J(x)$ the supremum, of 
$\abs{V\big(r(x) \big)}$ over Jacobi fields V along the geodesic $\gamma^{o,x}$ that satisfy 
$V_0=0, \, \abs{V_0'}=1$ and $V_0'\bot \dot{\gamma}_0^{o,x}$. 

Next we show that the integral appearing in Lemma \ref{phiintfinlem} is finite provided
the upper bound \eqref{curvupbound} and the pinching condition \eqref{pinchcond} for the 
sectional curvatures.
\begin{lem}\label{fsmall}
 Let $M$ be a Cartan-Hadamard manifold satisfying
     $$
       K(P) \le -\frac{1+\ve}{r(x)^2 \log r(x)},
     $$
  where $K(P)$ is the sectional curvature of any plane $P\subset T_xM$ that contains the
  radial vector field $\nabla r(x)$ and $x$ is any point in $M\setminus B(o,R_0)$.
  Then there exists $F \in \F_p$ such that
     $$
       F\left(\frac{r(x)}{c_1j(x)}\right)\Big(\log(1+r)+ \mathcal{C}(r)\Big)j(x)^{C(n-1)}
       \le r(x)^{-2}
     $$
 for any positive constant $C$ and for every $x\in M$ outside a compact set.
 \end{lem}

 \begin{proof}
Fix $\ve_0 \in (0,1)$ and denote $\lambda\coloneqq 1+\ve_0$. Then by
\eqref{fexists} there exists $F\in\F_p$ such that
\begin{equation*}
     F(t) \le \exp\left( -\tfrac{1}{t}\Big(\log\Big(e + \tfrac{1}{t}\Big)
     \Big)^{-\lambda} \right)
   \end{equation*}
 for all small $t$. Hence the claim follows if we show that
 $$
 \exp\left( -\frac{c_1 j(x)}{r(x)} \left( \log \left( e+\frac{c_1 j(x)}{r(x)} \right)
 \right)^{-\lambda} \right) \Big( \log \big(1+r(x)\big) + \mathcal{C}(r) \Big) j(x)^{C(n-1)}
 \le r(x)^{-2},
 $$
which, by taking logarithms, is equivalent with
\begin{align*}
\frac{c_1 j(x)}{r(x)} &\left( \log \left( e +\frac{c_1 j(x)}{r(x)} \right)
\right)^{-\lambda} -\log \Big( \log\big(1+r(x)\big) + \mathcal{C}(r) \Big) \\ &-C(n-1) \log j(x) -
2\log r(x) \ge 0.
\end{align*}
Let $\tve \in (0,\ve)$. Then the curvature upper bound and Proposition \ref{choiprop} 
implies that $j(x) \ge r(x)
\big(\log r(x)\big)^{1+\tilde{\ve}}$
for $r(x)\ge R_1>R_0$, so it is enough to show that
$$
f(t) \coloneqq \frac{c_1 t}{a} \left( \log \left( e+ \frac{c_1 t}{a} \right)
\right)^{-\lambda} -\log\Big(\log(1+a) + \mathcal{C}(a) \Big) - C(n-1) \log t - 2\log a
\ge 0
$$
 for all $t\ge a\big(\log a\big)^{1+\tve}$ when $a$ is big enough. By straight computation
 we get
 \begin{align*}
 f'(t) &= \left( \log \Big( e+\frac{c_1 t}{a} \Big) \right)^{-\lambda} \left(
 \frac{-\lambda c_1^2 t}{a^2 \log (e+ c_1 t/a)(e+c_1 t/a)} +\frac{c_1}{a}  \right) -
 \frac{C(n-1)}{t} \\
 &=  \frac{ \frac{c_1}{a}\left( 1- \frac{\lambda}{\log(e+ c_1 t/a)((ea)/(c_1 t)+1)}
 \right)}{\left(\log\big(e+\frac{c_1 t}{a}\big) \right)^{\lambda}} - \frac{C(n-1)}{t},
 \end{align*}
so noticing that $c_1t/a \ge c_1(\log a)^{1+\tve}$, 
which can be made big by increasing $a$,
and $(\log (e+ c_1t/a))^{\lambda} \le k (t/a)^{\nu}$, where $k$ is some constant and $\nu>0$
can be made as small as we wish, we obtain
$$
f'(t) \ge \frac{k}{a^{1-\nu} t^{\nu}} - \frac{C(n-1)}{t} \ge 0
$$
for all $t\ge a(\log a)^{1+\tve}$ when $a$ is big
enough. Finally we have to check that $f$ is positive at least when $t$ is big enough. To
see this, we notice that
\begin{align*}
f\big(a(\log a)^{1+\tve}\big) &= c_1(\log a)^{1+\tve} \left(
\log\big(e+c_1(\log a)^{1+\tve} \big) \right)^{-\lambda} \\
 &- \log \big(\log(1+a)-\mathcal{C}(a) \big) - C(n-1)\log \big((a\log a)^{1+\tve} \big)
 -2\log a,
\end{align*}
and this being positive is equivalent to
\begin{align*}
c_1(\log a)^{1+\tve} \ge \big(\log &\big( e +c_1(\log a)^{1+\tve}\big)  \big)^{\lambda}
\Big( \big(C(n-1)+2\big)\log a \\&+ \log\big( \log (1+a) + \mathcal{C}(a) \big)
+ C(n-1)\log(\log a)^{1+\tve}  \Big),
\end{align*}
which holds true for $a$ big enough, since $(\log a)^{1+\tve}$ increases
faster than $(\log a)\big(\log(e+c_1(\log a)^{1+\tve})^{\lambda}\big)$.
 \end{proof}

To prove the Theorem \ref{mainthm}, we give a proof for the following  localized version
that shows the $\A$-regularity of a point $x_0\in \pinf M$. That, in turn, implies
Theorem \ref{mainthm} since the uniqueness follows from the comparison principle.

The proof of the following theorem is the same as the proof of
\cite[Theorem 17]{casterasholopainenripoll2} except that to prove
 $$
   \int_{\Omega} F\left(\frac{c_0|\nabla\theta| r \log(1+r)}{L(r)}  \right) L(r) < \infty,
 $$
where $L(r) = \log(1+r) + \mathcal{C}(r)$,
we use Lemma \ref{fsmall} instead of some estimates involving the curvature lower bound.
For convenience, we will also write down the proof.
\begin{theo}
Let $M$ be a Cartan-Hadamard manifold of dimension $n\ge2$ and let $x_0\in\pinf M$.
Assume that $x_0$ has a cone neighborhood $U$ such that
\begin{equation}\label{curvupbound1}
K(P) \le -\frac{1+\ve}{r(x)^2\log r(x)},
\end{equation}
for some constant $\ve>0$, where $K(P)$ is the sectional curvature of any two-dimensional
subspace $P\subset T_xM$ containing the radial vector $\nabla r(x)$, with $x\in U \cap M$.
Suppose also that there exists a constant $C_K<\infty$ such that
\begin{equation}\label{pinchcond1}
|K(P)| \le C_K |K(P')|
\end{equation}
whenever $x\in U \cap M$ and $P,P'\subset
T_xM$ are two-dimensional subspaces
containing the radial vector $\nabla r(x)$. Then $x_0$ is $\A$-regular for every
$\A\in\A^p(M)$ with $1< p < n\alpha/\beta$.
\end{theo}

\begin{proof}
Let $f\colon \pinf M \to \R$ be a continuous function. To prove the $\A$-regularity of $x_0$,
we need to show that
   $$
     \lim_{x\to x_0} \overline{H}_f(x) = f(x_0).
   $$
Fix $\ve'>0$ and let $v_0 = \dot{\gamma}_0^{o,x_0}$ be the initial vector of the geodesic
ray from $o$ to $x_0$. Furthermore, let $\delta \in (0,\pi)$ and $R_0>0$ be such that
$T(v_0,\delta,R_0) \subset U$ and that $|f(x_1)-f(x_0)| < \ve'$ for all $x_1 \in
C(v_0,\delta) \cap \pinf M$. Fix also $\tve \in (0,\ve)$, where $\ve$ is the constant in
\eqref{curvupbound1}, and let $r_1 > \max(2,R_1)$, where $R_1 \ge R_0$ is given by
Proposition \ref{choiprop}.

We denote $\Omega = T(v_0,\delta,r_1) \cap M$ and define $\theta \in C(\bM)$ by setting
   $$
     \theta(x) = \min \Big(1, \max \big( r_1 + 1-r(x),\delta^{-1} \an_o(x_0,x)\big) \Big).
   $$

Let $\Omega_j = \Omega \cap B(o,j)$ for integers $j>r_1$ and let $u_j$ be the unique $\A$-harmonic
function in $\Omega_j$ with $u_j - \theta \in \Wop(\Omega_j)$. Each $y\in \p\Omega_j$ is
$\A$-regular and hence functions $u_j$ can be continuously extended to $\p\Omega_j$ by
setting $u_j = \theta$ on $\p\Omega_j$. Next we notice that $0\le u_j \le 1$, so the
sequence $(u_j)$ is equicontinuous, and hence, by Arzelá-Ascoli, we obtain a subsequence
(still denoted by $(u_j)$) that converges locally uniformly to a continuous function 
$u\colon \bar{\Omega} \to [0,1]$. It follows that $u$ is $\A$-harmonic in $\Omega$; see e.g.
\cite[Chapter 6]{heinonenkilpelainenmartio}. 

Next we aim to prove  that
   \begin{equation}\label{limitu}
    \lim_{\underset{x\in\Omega}{x\to x_0}} u(x) =0,
   \end{equation}
and for that we use geodesic polar coordinates $(r,v)$ for points $x\in\Omega$. Here $r=r(x)\in[r_1,\infty)$
and $v=\dot{\gamma}^{o,x}_0$, and we denote by $\lambda(r,v)$ the Jacobian of these polar
coordinates. Denote $\tilde{\theta} = \theta/c_0$, $\tilde{u}_j = u_j/c_0$ and $\tilde{u} =
u/c_0$, where $c_0\ge1$ is a constant given by Lemma \ref{phiintfinlem}.
Then applying Fatou's lemma and Lemma \ref{phiintfinlem} to $\Omega_j$ we
obtain
\begin{align}\label{phiintfin}
\int_{\Omega} \varphi &\big( |\tilde{u}-\tilde{\theta}|\big)^p = \int_{\Omega} \varphi\big(
     |u-\theta|/c_0\big)^p \le \liminf_{j\to\infty} \int_{\Omega_j} \varphi\big(
     |u_j-\theta|/c_0\big)^p \nonumber \\
&\le \liminf_{j\to\infty} \int_{\Omega_j} \varphi\big(
     |u_j-\theta|/c_0\big)^p L(r) \nonumber \\
&\le c_0 + c_0 \int_{\Omega} F\left(\frac{c_0 |\nabla \theta| r \log(1+r)}{L(r)}\right) L(r) \nonumber\\
&= c_0 + c_0 \int_{r_1}^{\infty} \int_{S_oM} F\left(\frac{c_0 |\nabla \theta(r,v)| r
\log(1+r)}{\log(1+r) + \mathcal{C}(r)}\right) \big(\log(1+r)  \\
     &\quad \quad+ \mathcal{C}(r)\big) \lambda(r,v) \,dv\,dr \nonumber \\
&\le c_0 + c_0 \int_{r_1}^{\infty} \int_{S_oM} F\left(\frac{r}{c_1j(r,v)}\right)
     \Big(\log(1+r)+ \mathcal{C}(r)\Big)j(r,v)^{C(n-1)} \,dv\,dr \nonumber \\
&< \infty \nonumber.
\end{align}
At the end we applied also Lemmas \ref{vahakpinclemma3}, \ref{anglegrad}, and \ref{fsmall}.

Next, we extend each $u_j$ to a function $u_j \in \LWp(M) \cap C(M)$ by setting $u_j(y) =
\theta(y)$ for every $y\in M\setminus \Omega_j$. Let $x\in \Omega$ and fix $s\in (0,r_S)$.
For $j$ large enough, we obtain (by \cite[Lemma 2.20]{vahakangasyoung}, stated also in
\cite[Lemma 15]{casterasholopainenripoll2})
 $$
   \sup_{B(x,s/2)} \varphi\big(|\tilde{u}_j - \tilde{\theta}| \big)^{p(n+1)} \le c \int_{B(x,s)}
   \varphi\big(|\tilde{u}_j - \tilde{\theta}| \big)^p.
 $$
Applying this with dominated convergence theorem, we get
   \begin{align} \label{supintestim}
    \sup_{B(x,s/2)} &\varphi\big(|\tilde{u} - \tilde{\theta}| \big)^{p(n+1)} =
    \sup_{B(x,s/2)} \lim_{j\to\infty} \varphi\big(|\tilde{u}_j - \tilde{\theta}| \big)^{p(n+1)} \nonumber \\
    &\le \limsup_{j\to\infty} \sup_{B(x,s/2)} \varphi\big(|\tilde{u}_j - \tilde{\theta}| \big)^{p(n+1)} \\
    &\le c\limsup_{j\to\infty} \int_{B(x,s)} \varphi\big(|\tilde{u}_j - \tilde{\theta}| \big)^{p}
    = c\int_{B(x,s)} \varphi\big(|\tilde{u} - \tilde{\theta}| \big)^{p}. \nonumber
   \end{align}
Let $(x_k)\subset \Omega$ be a sequence such that $x_k\to x_0$ as $k\to\infty$. We apply the
estimate \eqref{supintestim} with $x=x_k$ and a fixed $s\in(0,r_S)$, together with
\eqref{phiintfin}, to obtain
   $$
     \lim_{k\to\infty} \sup_{B(x_k,s/2)} \varphi\big(|\tilde{u} - \tilde{\theta}| \big)^{p(n+1)}
     \le c \lim_{k\to\infty} \int_{B(x_k,s)} \varphi\big(|\tilde{u} - \tilde{\theta}| \big)^{p} =0.
   $$
It follows that
   $$
     \lim_{k\to\infty} |\tilde{u}(x_k) - \tilde{\theta}(x_k)| = 0,
   $$
which, in turn, implies \eqref{limitu}.

Define a function $w\colon M \to \R$ by
 $$
   w(x) = \begin{cases}
           \min \big(1, 2u(x)\big) &\text{if } x\in\Omega; \\
           1, &\text{if } x\in M\setminus \Omega.
          \end{cases}
 $$
The minimum of two $\A$-superharmonic functions is $\A$-superharmonic and hence $w$ is
$\A$-superharmonic. The definition of $\overline{H}_f$ implies that
 $$
   \overline{H}_f \le f(x_0) + \ve' + 2(\sup|f|) w,
 $$
and therefore, by \eqref{limitu}, we have
 $$
   \limsup_{x\to x_0} \overline{H}_f(x) \le f(x_0) + \ve'.
 $$
Similarly one can prove that
 $$
   \liminf_{x\to x_0} \underline{H}_f(x) \ge f(x_0) - \ve',
 $$
and because $\overline{H}_f \ge \underline{H}_f$ and $\ve'$ was arbitrary, we conclude
that
 $$
   \lim_{x\to x_0} \overline{H}_f(x) = f(x_0).
 $$
Therefore $x_0$ is $\A$-regular point.
\end{proof}

\bibliography{minimal}
\bibliographystyle{plain}

\end{document}